\documentclass[12pt, reqno]{amsart}
\usepackage{amssymb}
\usepackage{amsmath}
\usepackage{amscd}
\usepackage{amssymb}
\usepackage{eucal}
\usepackage{amsmath}
\usepackage{amscd}
\usepackage[dvips]{color}
\usepackage{multicol}
\usepackage[all]{xy}           
\usepackage{graphicx}
\usepackage{color}
\usepackage{colordvi}
\usepackage{xspace}
\usepackage{tikz}

\usepackage{ifpdf}
\ifpdf
 \usepackage[colorlinks,final,backref=page,hyperindex]{hyperref}
\else
 \usepackage[colorlinks,final,backref=page,hyperindex,hypertex]{hyperref}
\fi

\newtheorem{theorem}{Theorem}[section]

\newtheorem{prop}[theorem]{Proposition}

\newtheorem{defn}[theorem]{Definition}
\newtheorem{lemma}[theorem]{Lemma}
\newtheorem{coro}[theorem]{Corollary}
\newtheorem{prop-def}{Proposition-Definition}[section]

\newtheorem{remark}[theorem]{Remark}

\newtheorem{exam}[theorem]{Example}

\allowdisplaybreaks

\newcommand{\C}{\mathbb{C}} 
 
\newcommand{\Z}{\mathbb{Z}} 
 
\newcommand{\F}{\mathbb{F}} 

\newcommand{\g}{\mathfrak{g}} 
\newcommand{\h}{\mathfrak{h}}

\newcommand{\gl}{\mathfrak{gl}} 
\newcommand{\sll}{\mathfrak{sl}}

\newcommand{\ext}{{\rm Ext}_{\omega}^{1}}
\newcommand{\TDer}{{\rm TDer}}
\newcommand{\Der}{{\rm Der}}

\newcommand{\ad}{{\rm ad}}

\newcommand{\Tr}{{\rm Tr}}
\newcommand{\Hom}{{\rm Hom}}
\newcommand{\GL}{{\rm GL}}

\newcommand{\A}{\mathcal{A}} 
\newcommand{\B}{\mathcal{B}} 
\newcommand{\D}{\mathcal{D}} 
\newcommand{\HH}{\mathcal{H}} 
\newcommand{\LL}{\mathcal{L}} 
\newcommand{\R}{\mathcal{R}}

\newcommand{\lam}{\lambda}

\begin{document}

\title{Representations of $\omega$-Lie Algebras and Tailed Derivations}

\author{Runxuan Zhang}
\address{School of Mathematics and Statistics, Northeast Normal University, Changchun 130024, P.R. China}
\email{zhangrx728@nenu.edu.cn}



\maketitle

\baselineskip=18pt

\begin{abstract}
We study the representation theory of finite-dimensional  $\omega$-Lie algebras over the complex field. We derive an $\omega$-Lie version of the classical Lie's theorem, i.e., any finite-dimensional irreducible module of a soluble $\omega$-Lie algebra is one-dimensional. We also prove that indecomposable modules of some three-dimensional $\omega$-Lie algebras could be parametrized by the complex field and nilpotent matrices. We introduce the notion of a tailed derivation of a nonassociative  algebra $\g$ and
prove that if $\g$ is a Lie algebra, then there exists a one-to-one correspondence between tailed derivations of $\g$ and
one-dimensional $\omega$-extensions of $\g$. 
\end{abstract}

\section{Introduction}	

In 2007, Nurowski introduced the notion of $\omega$-Lie algebras for which the original motivation stems from some geometry considerations, see \cite{Nur07}, \cite{BN07} and  \cite{Nur08}. More specifically,
a vector space $L$ over a field $\F$  equipped with a skew-symmetric bracket $[-,-]: L\times L \longrightarrow L$  and a bilinear form $\omega: L\times L \longrightarrow \F$ is called an \textbf{$\omega$-Lie algebra} provided that
\begin{equation}
[[x,y],z]+ [[y,z],x]+  [[z,x],y]=\omega(x,y)z+\omega(y,z)x+\omega(z,x)y \tag{$\omega$-Jacobi identity}
\end{equation}
for all $x,y,z\in L$. Clearly, $\omega$-Lie algebras with $\omega=0$ are nothing but ordinary Lie algebras, which means that the notion of $\omega$-Lie algebras extends that of Lie algebras.

The present article is devoted to a study of the representation theory of  finite-dimensional  $\omega$-algebras over the complex field.
Let's recall some development on this subject.
 In 2010, Zusmanovich in \cite[Section 9, Theorem 1]{Zus10} proved an important result on the structure of  $\omega$-Lie algebras, which says that all finite-dimensional non-Lie  $\omega$-Lie algebras are either low-dimensional or have a quite degenerate structure.  By the $\omega$-Jacobi identity one sees that there are no non-Lie $\omega$-Lie algebras of dimensions one and two.
In our previous works \cite{CLZ14} and \cite{CZ17}, we  derived a rough classification of three- and four-dimensional complex $\omega$-Lie algebras. With the classification, we recently calculated the automorphism groups and the derivation algebras of low-dimensional  $\omega$-Lie algebras  over the complex field, reformulated elementary facts about the representation theory of  $\omega$-Lie algebras, and we also proved that all finite-dimensional irreducible representations of the family $C_{\alpha}$ of $\omega$-Lie algebras are one-dimensional; see \cite[Sections 6 and 7]{CZZZ18}.

The first purpose of this article is to generalize the classical Lie's theorem of complex soluble Lie algebras to the case of $\omega$-Lie algebras. We introduce the following notion of  degree of  $\omega$-Lie algebras. 

\begin{defn}{\rm
Suppose that $L$ is a finite-dimensional $\omega$-Lie algebra. The positive integer
$\deg(L):=\min\{\dim(L)-\dim(I)\mid I\subset L\textrm{ is a proper ideal}\}$ is called the \textbf{degree} of  $L$.
}\end{defn}

We will show that soluble $\omega$-Lie algebras are  of degree 1; see Proposition \ref{prop2.2} below.
Our first main result can be formulated as follows.

\begin{theorem}\label{mainthm2}
Let $L$ be a non-simple complex $\omega$-Lie algebra of degree 1 with a soluble ideal $\g$ of maximal dimension $\dim(L)-1$ and  $V$ be  a finite-dimensional irreducible $L$-module. Then $\dim(V)=1$.
\end{theorem}

 Proposition \ref{prop2.2} and Theorem \ref{mainthm2} combine to a direct consequence which could be regarded as an $\omega$-Lie version of the classical Lie's theorem.

\begin{coro}[The $\omega$-Lie version of Lie's theorem]\label{coro1.3}
Let $L$ be a   finite-dimensional   soluble $\omega$-Lie algebra  over the complex field and  $V$ be  a finite-dimensional irreducible $L$-module. Then $\dim(V)=1$.
\end{coro}

We also give some applications of Theorem \ref{mainthm2} and fundamental properties of $\omega$-Lie algebra modules in Section \ref{sec2}.

The second goal of this paper is to study indecomposable representations of some three-dimensional non-Lie $\omega$-Lie algebras.
Note that we have already classified these $\omega$-Lie algebras in \cite[Theorem 2]{CLZ14} into $\LL:=\{L_{1},L_{2}, A_{\alpha},B,C_{\alpha}\}$, see Section \ref{sec2} for details.
Let $L\in\{L_{1},A_{\alpha}\}$  and $\R_{n}(\C)$ be the set of all indecomposable $L$-modules on $\C^{n}$. Section \ref{sec3} is devoted to a proof of the following second main result.

\begin{theorem}\label{mainthm3}
The equivalence classes in $\R_{n}(\C)$ could be parametrized by
the complex field $\C$, the conjugacy classes of $n\times n$ nilpotent matrices and an affine variety. 
\end{theorem}

Our third purpose is to study one-dimensional $\omega$-extensions of Lie algebras. Note that  one-dimensional extensions of a Lie algebra $\g$ can be parameterized by the set of all twisted derivations of $\g$; see  \cite[Proposition 5.4]{AM14}.
Let $\g$ be a Lie algebra and $L=\g\oplus \C x$ be the vector space of dimension $\dim(\g)+1$.
Then $L$ is called a one-dimensional $\omega$-\textbf{extension} of $\g$ through $\C x$ if there exists an $\omega$-Lie algebra structure on $L$ containing $\g$ as an ideal  and $\omega(\g,\g)=0$. To describe the set $\ext(\g)$  of all one-dimensional $\omega$-extensions of $\g$,
we introduce the notion of tailed derivations of  nonassociative algebras.

\begin{defn}{\rm
Let $A$ be a nonassociative algebra. A linear map $D:A\longrightarrow A$  is called a \textbf{tailed derivation} of $A$ if there exists a linear form $d: A\longrightarrow \F~ (y\mapsto d_{y})$ such that
\begin{equation}
D([y,z])=[D(y),z]+[y,D(z)]+d_{z}y-d_{y}z \label{ }
\end{equation}
for all $y,z\in A$.
}\end{defn}

We observe  that for a tailed derivation $D$, such linear form $d$ is unique; and moreover,  in \cite[Section 6, Definition]{Zus10}, tailed derivations of an anti-commutative algebra  have appeared as
a special kind of $(\alpha,\lambda)$-derivations with $\lambda=0$.
Clearly, all derivations of $A$ are tailed derivations with trivial tails, i.e., $d_{y}=d_{z}=0$ for all $y,z\in A$.
We denote by $\TDer(A)$ the set of all tailed derivations of $A$. We will show that $\TDer(A)$ is a Lie subalgebra of the general linear Lie algebra $\gl(A)$; see Proposition \ref{prop4.1}. Thus $\Der(A)\subseteq \TDer(A)\subseteq\gl(A)$ as Lie subalgebras, with the containment might be strict; see Example \ref{exam4.2}.
Now the third main result can be stated as follows.

\begin{theorem} \label{mainthm4}
Let $\g$ be a Lie algebra and $\omega$ be  a skew-symmetric bilinear form on $\g\oplus \C x$. Then
there exists a one-to-one correspondence between $\ext(\g)$ and $\TDer(\g)$.
\end{theorem}

We also provide an example that demonstrate that $\omega$-Lie algebras could be constructed 
by Lie algebras and their tailed derivations; see Example \ref{exam4.6}.\\

\noindent{\bf Conventions.} The Lie algebra notions that do not involve the form $\omega$ in their definitions are extended verbatim to $\omega$-Lie algebras: for example,  subalgebras, ideals, simple, soluble and abelian algebras.

Throughout this article  we assume that the ground field is the complex field $\C$. All representations (modules), vector spaces and  algebras are finite-dimensional over $\C$. We use $z_{V}$ to denote the linear transformation of an abstract element $z$ acting on a vector space $V$. We use $\Z^+$ and $\Z_{\geqslant 0}$ to denote the sets of  positive and  non-negative integers, respectively.

\section{The $\omega$-Lie version of Lie's theorem}\label{sec2}

In this section, we  show Theorem \ref{mainthm2} and provide some applications. To begin with, we present two examples of non-simple  $\omega$-Lie algebras.

\begin{exam}\label{exam2.1}
{\rm The following three-dimensional $\omega$-Lie algebras are   of degree 1:
\begin{enumerate}
  \item $L_{1}:[x,z]=0,[y,z]=z, [x,y]=y\textrm{ and }\omega(y,z)=\omega(x,z)=0,\omega(x,y)=1\,;$
  \item $L_{2}:[x,y]=0,[x,z]=y,[y,z]=z \textrm{ and } \omega(x,y)=0, \omega(x,z)=1, \omega(y,z)=0\,.$
\end{enumerate}
Here $\{x,y,z\}$ denotes a basis of the underlying vector space. We observe that the subspace spanned by $y$ and $z$ is a proper ideal, so
$L_{1}$ and $L_{2}$ are non-simple and  of degree 1.}\end{exam}

Note that $L_{1}$ and $L_{2}$ in Example \ref{exam2.1}  are both soluble $\omega$-Lie algebras. Moreover, we have the following more general result.

\begin{prop}\label{prop2.2}
Soluble $\omega$-Lie algebras are  of degree 1.
\end{prop}

\begin{proof}
Let $L$ be an $n$-dimensional  soluble $\omega$-Lie algebra. Then $[L,L]\neq L$ and so it is not simple. To show that $L$ has degree 1, we may find an $(n-1)$-dimensional subspace $I$ of $L$ such that $[L,L]\subseteq I\subset L$. As $[I,L]\subseteq [L,L]\subseteq I$, we see that $I$ is an ideal of $L$. Clearly, $I$ is a proper ideal with the maximal dimension $n-1$. Hence, $L$ has degree 1.
\end{proof}


We also present some examples of three-dimensional simple $\omega$-Lie algebras.

\begin{exam}{\rm
Let $\{x,y,z\}$ be a basis of $\C^{3}$. The following $\omega$-Lie algebras are simple:
\begin{enumerate}
  \item
  $A_{\alpha}: [y,z]=z,[x,z]=y-z, [x,y]=x+\alpha z,
\omega(y,z)=\omega(x,z)=0,
\omega(x,y)=-1\,;$
  \item  $B:  [y,z]=z, [x,y]=z-x, [x,z]=y, \omega(y,z)=\omega(x,y)=0, \omega(x,z)=2\,;$
     \item $C_{\alpha}:  [y,z]=z, [y,x]=\alpha x, [z,x]=y, \omega(y,z)=\omega(x,y)=0,
\omega(z,x)=1+\alpha\,,$
\end{enumerate}
 where  $\alpha\in \C$. See \cite[Proposition 7.1]{CZ17} for the details. Comparing with \cite[Theorem 2]{CLZ14} or \cite[Theorem 1.4]{CZ17},
we see that in this example the generating relations actually have been reformulated by choosing a suitable basis. 
 }\end{exam}

\begin{remark}{\rm
In fact, \cite[Theorem 2]{CLZ14} indicates that
every three-dimensional  non-Lie $\omega$-Lie algebra over $\C$ must be isomorphic to one of $\LL=\{L_{1},L_{2}, A_{\alpha},B,C_{\alpha}\}$.
}\end{remark}

Here we  provide an example of four-dimensional non-simple $\omega$-Lie algebra of degree $>1$.

\begin{exam}{\rm  Let $\{x,y,z,e\}$ be a basis of $\C^{4}$. In the following $\omega$-Lie algebra
\begin{eqnarray*}
\widetilde{B} :&& [x,y]=y,[x,z]=y+z, [y,z]=x,  [e,x]=-2e, [e,y]=0,[e,z]=0\,, \\
&\textrm{and } &\omega(x,y)=\omega(x,z)=0,\omega(y,z)=2, \omega(e,x)=\omega(e,y)=\omega(e,z)=0\,,
\end{eqnarray*}
the subspace spanned by $\{e\}$ is a proper ideal of $\widetilde{B}$ with the maximal dimension 1, i.e., there are no proper ideas in
$\widetilde{B}$ with dimension $>1$. Hence $\widetilde{B}$ is a non-simple $\omega$-Lie algebra of degree 3. 
}\end{exam}

Let $L$ be an $\omega$-Lie algebra  and $V$ be a finite-dimensional vector space.
Recall that $V$ is called an \textbf{$L$-module} if there exists a bilinear map $L\times V\longrightarrow V, (x,v)\mapsto x\cdot v$ such that 
\begin{equation}
[x,y] \cdot v= x\cdot(y \cdot v)-y\cdot(x\cdot v)+\omega(x,y)v \label{eq2.1}
\end{equation}
for all $x,y\in L$ and $v\in V.$

To derive an $\omega$-Lie version of the classical Lie's theorem, we concentrate on the class of non-simple $\omega$-Lie algebras of degree 1, and we  give a proof of Theorem \ref{mainthm2}.

\begin{proof}[Proof of Theorem \ref{mainthm2}]
If $\dim(L)\leqslant2$, then $L$ is a soluble Lie algebra. It follows from the classical Lie's theorem that $\dim(V)=1$.
Thus we may suppose $\dim(L)\geqslant 3$ and regard $V$ as a $\g$-module.  By  \cite[Corollary 3.2]{Zus10}, we see that a proper soluble ideal 
$\g$ of $L$  is a soluble Lie algebra. If $V$ is an irreducible $\g$-module, then classical Lie's theorem implies $\dim(V)=1$, and we are done.

Now we assume that $V$ is a reducible $\g$-module and there exists an irreducible $\g$-submodule $W\subset V$.
Applying the classical Lie's theorem again we see that $\dim(W)=1$. Fix a nonzero vector $w_{0}\in W$, there exists a one-dimensional representation $\lam$ of $\g$ given by $W$  such that $g\cdot w_{0}=\lam(g)w_{0}$ for all $g\in \g$. Define
\begin{equation}
U:=\{v\in V\mid g\cdot v=\lam(g)v\textrm{ for all }g\in \g\}\,.\label{ }
\end{equation}
Then $W\subseteq U\subseteq V.$ We claim that $U$ is also an $L$-module. If this claim holds, the irreducibility of $V$ as an $L$-module, implies that $V=U$; thus $g\cdot v=\lam(g)v$ for all $g\in\g$ and $v\in V$.
Moreover, for any vector $\ell\in L$ but not in $\g$, let $J$ denote the one-dimensional subspace spanned by $\ell$. Then $L$ can be decomposed into the direct sum $\g\oplus J$ as  vector spaces. Let $v_{0}$ be an eigenvector of $\ell_{V}$ and let $V_{0}$
denote the one-dimensional subspace spanned by $v_{0}$. Then  $\ell \cdot v_{0} \in V_{0}\subseteq V$, which together with
the fact that $g\cdot v=\lam(g)v$ for all $g\in\g$ and $v\in V$, implies that $V_{0}$ is an $L$-submodule of $V$. As $V$ is irreducible, we have $V=V_{0}$. Hence, $\dim(V)=\dim(V_{0})=1$.

Therefore, to accomplish the proof, it is sufficient to prove the claim  that $U$ is an $L$-module.
For all $g,g'\in\g$ and $v\in U$, we see that $g\cdot(g' \cdot v)=g\cdot(\lam(g')v)=\lam(g')(g \cdot v)=\lam(g')\lam(g) v=
\lam(g)\lam(g') v=\lam(g)(g' \cdot v)$, i.e., $g'\cdot U\subseteq U$ for all $g'\in\g$. Thus it suffices to show that
$\ell \cdot U\subseteq U$; in other words,  we have to prove
that $g\cdot(\ell\cdot v)=\lam(g)(\ell \cdot v)$ for all $g\in\g$ and $v\in U$.
As $[\g,J]\subseteq\g$, we see that $$\lam([g,\ell]) v=[g,\ell]\cdot  v=g\cdot (\ell\cdot v)-\ell\cdot(g\cdot v)+\omega(g,\ell)v=g\cdot(\ell\cdot v)-\lam(g)(\ell\cdot v)+\omega(g,\ell)v\,.$$ Thus it suffices to show that
\begin{equation}
\lam([g,\ell])=\omega(g,\ell)\,. \label{ }
\end{equation}
To do this, we let $0\neq u\in U$ and define $u_{i}:=\ell\cdot u_{i-1}$ for $i\in\Z^{+}$, starting with $u_{0}:=u$ and $u_{1}:=\ell\cdot u$. Let $V'$ be the subspace spanned by $\{u_{i}\mid i\in\Z_{\geqslant 0}\}$. Since $V'\subseteq V$ and $\dim(V)$ is finite, there exists some $k\in \Z_{\geqslant 0}$ such that $V'$ has a basis $\{u_{0},u_{1},\dots,u_{k}\}$. Clearly, $\ell\cdot V'\subseteq V'$. Let $V'_{j}$ denote the subspace spanned by $u_{0},u_{1},\dots,u_{j}$ for $j=0,1,\dots,k$. Induction on $j$ shows that
$g\cdot u_{j}-\lam(g)u_{j}\in V_{j-1}'$ for all $g\in\g$. This means that $V'$ is an $L$-submodule of $V$.
As $V$ is irreducible, we have $V=V'$, and the resulting matrix $g_{V}$ can be written as an upper triangular matrix with the diagonals $\lam(g)$. Thus $\Tr(g_{V})=(k+1)\lam(g)$ for all $g\in\g$; in particular,
$\Tr([g,\ell]_{V})=(k+1)\lam([g,\ell])$. Since $[g,\ell]_{V}=g_{V}\circ\ell_{V}-\ell_{V}\circ g_{V}+\omega(g,\ell)1$, it follows that
$\Tr([g,\ell]_{V})=\Tr(\omega(g,\ell)1)=(k+1)\omega(g,\ell).$ This implies that $\lam([g,\ell])=\omega(g,\ell)$ and the proof is completed.
\end{proof}

We provide two applications of Theorem \ref{mainthm2}.

\begin{proof}[Proof of Corollary \ref{coro1.3}]
As any ideal of a soluble $\omega$-Lie algebra is soluble,  this corollary  could be obtained directly from Theorem \ref{mainthm2} and Proposition \ref{prop2.2}.
\end{proof}

Recall that an $\omega$-Lie algebra $L$ is said to be \textbf{multiplicative} if there exists a linear form $\lam:L\longrightarrow \C$ such that
$\omega(x,y)=\lam([x,y])$ for all $x,y\in L$; see \cite[Section 2]{Zus10} and \cite[Section 6]{CZZZ18} for more results on multiplicative $\omega$-Lie algebras.

\begin{lemma}
Let $L$ be an $\omega$-Lie algebra. Then $\ker(\omega)=\{x\in L\mid \omega(x,y)=0  \textrm{ for all }y\in L\}$ is an $L$-module via the adjoint action.
\end{lemma}

\begin{proof}
Indeed, for all $x\in\ker(\omega)$ and $y,z\in L$, the $\omega$-Jacobi identity gives
$[[y,z],x]+[[z,x],y]+[[x,y],z]=\omega(y,z)x+\omega(z,x)y+\omega(x,y)z=\omega(y,z)x$. Then
$[[y,z],x]=[y,[z,x]]-[z,[y,x]]+\omega(y,z)x$ and hence $\ker(\omega)$ is an $L$-module.
\end{proof}

\begin{prop}
Let $L$ be a non-simple $\omega$-Lie algebra of degree 1 with a soluble ideal $\g$ of maximal dimension $\dim(L)-1$. If $\dim(L)>2$, then $L$ is multiplicative.
\end{prop}

\begin{proof}
As $\dim(L)>2$, it follows from \cite[Lemma 8.1]{Zus10} that $\omega$ is degenerate. Then $\ker(\omega)$ is a nonzero $L$-module.
Let $W$ be an irreducible
$L$-submodule of $\ker(\omega)$. By Theorem \ref{mainthm2} we see that $\dim(W)=1$. It follows from \cite[Lemma 2.1]{Zus10} that $L$ is multiplicative.
\end{proof}

We give some remarks on modules and cohomology of $\omega$-Lie algebras. 
We refer to \cite[Section 6]{CZZZ18} for some fundamental  properties of modules for $\omega$-Lie algebras. The following example shows that  the cohomology groups $\HH^{n}(L,V)$ of an  $\omega$-Lie algebra $L$ with coefficients in an $L$-module $V$ cannot be defined by the same formula for the differential as for ordinary Lie algebras via the way of Chevalley–Eilenberg complex; compared with \cite{CE48}.

\begin{exam}{\rm  
Suppose $L$ is an $\omega$-Lie algebra and $V$ is an $L$-module. As in the Chevalley–Eilenberg complex, we define  the $\C$-vector space of $k$-cochains of $L$ with coefficients in $V$ to be $C^0(L,V):=V$ and $C^k(L,V):=\Hom_{\C}(\land^{k}L,V)$ for $k\geqslant 1$. The 
differential $d_k:C^k(L,V)\longrightarrow C^{k+1}(L,V)$ is defined as 
\begin{eqnarray*}
 d_k(f)(x_1,\dots,x_{k+1})& = &\sum_{i=1}^{k+1}(-1)^{i+1}x_i\cdot f(x_1,\dots,\widehat{x_i},\dots,x_{k+1})  \\
 &&  +\sum_{1\leqslant i<j\leqslant k+1}(-1)^{i+j}f([x_i,x_j],x_1,\dots,\widehat{x_i},\dots,\widehat{x_j},\dots,x_{k+1})\,.
\end{eqnarray*}
In particular, if $v\in C^0(L,V)=V$, then $d_0(v):L\longrightarrow V$ is given by $d_0(v)(x)=x\cdot v$ for all $x\in L$. For 
$f\in C^1(L,V)$, $d_1(f)\in C^2(L,V)$ is given by 
$$d_1(f)(x,y)=x\cdot f(y)-y\cdot f(x)-f([x,y])$$
 for all $x,y\in L$.
We observe that the map $d_1\circ d_0$ is not zero, unless $L$ is a Lie algebra. In fact, for $v\in V$ and $x,y\in L$,
\begin{eqnarray*}
(d_1\circ d_0)(v)(x,y) & = & d_1(d_0(v))(x,y) \\
 & = & x\cdot d_0(v)(y)-y\cdot d_0(v)(x)-d_0(v)([x,y])\\
 &=& x\cdot (y\cdot v)-y\cdot (x\cdot v)-[x,y]\cdot v\\
 &=&-\omega(x,y)v\,.
\end{eqnarray*}
The last equality follows from Eq. (\ref{eq2.1}). 
}\end{exam}

Moreover, let $L$ be an $\omega$-Lie algebra and $V, W$ be two $L$-modules. We also note that unlike the situation of ordinary Lie algebras, the map defined by $$(x,v\otimes w)\mapsto x\cdot v\otimes w+v\otimes x\cdot w$$ would not give an $L$-module structure on the tensor product $V\otimes W$, where $x\in L, v\in V$ and $w \in W$. However, 
for multiplicative $\omega$-Lie algebras we have the following proposition.

\begin{prop}
Let $L$ be a multiplicative $\omega$-Lie algebra with the linear form $\lam$ and $V, W$  be $L$-modules. Then
$V\otimes W$ is an $L$-module defined by
\begin{equation}
x\cdot (v\otimes w):=x\cdot v\otimes w+v\otimes x\cdot w-\lam(x)v\otimes w\,,\label{ }
\end{equation}
where $x\in L, v\in V$, and $w\in W$.
\end{prop}

\begin{proof}
For an arbitrary element $y\in L$,  we have
$[y,x] \cdot (v\otimes w)=[y,x]\cdot v\otimes w+v\otimes [y,x]\cdot w-\lam([y,x])v\otimes w=y\cdot(x\cdot v)\otimes w-x\cdot (y\cdot v)\otimes w+\omega(y,x)v\otimes w+ v\otimes (y\cdot (x\cdot w))- v\otimes (x\cdot(y\cdot w))+ \omega(y,x)v\otimes w-\lam([y,x])v\otimes w=y\cdot (x\cdot v)\otimes w-x\cdot(y\cdot v)\otimes w+ v\otimes (y\cdot(x\cdot w))- v\otimes (x\cdot(y\cdot w))+\omega(y,x)v\otimes w.$ The last equality holds by the fact that $\lam([y,x])=\omega(y,x)$. On the other hand, we have
\begin{eqnarray*}
y\cdot(x\cdot(v\otimes w))&=&y\cdot(x\cdot v\otimes w+v\otimes x\cdot w-\lam(x)v\otimes w) \\
&=&y\cdot(x\cdot v)\otimes w+x\cdot v\otimes y\cdot w-\lam(y)x\cdot v\otimes w\\&&+y\cdot v\otimes x\cdot w+v\otimes y\cdot(x\cdot w)-\lam(y)v\otimes x\cdot w\\
&&-\lam(x)y\cdot v\otimes w-\lam(x)v\otimes y\cdot w+\lam(x)\lam(y)v\otimes w\,,\\
x\cdot (y\cdot (v\otimes w))&=&x\cdot(y\cdot v\otimes w+v\otimes (y\cdot w)-\lam(y)v\otimes w) \\
&=&x\cdot(y\cdot v)\otimes w+y\cdot v\otimes x\cdot w-\lam(x)y\cdot v\otimes w\\&&+x\cdot v\otimes y\cdot w+v\otimes x\cdot(y\cdot w)-\lam(x)v\otimes y\cdot w\\
&&-\lam(y)x\cdot v\otimes w-\lam(y)v\otimes x\cdot w+\lam(y)\lam(x)v\otimes w\,.
\end{eqnarray*}
Then $y\cdot (x\cdot (v\otimes w))-x\cdot (y\cdot (v\otimes w))+\omega(y,x)v\otimes w=[y,x]\cdot (v\otimes w)$, which implies that
$V\otimes W$ is an $L$-module.
\end{proof}


Note that the adjoint map does not give an $L$-module structure on $L$, unless $L$ is a Lie algebra.
The following example demonstrates that for $k\in\Z^+$, the space of $k$-cochains $C^{k}(L,V)$ might not be an $L$-module via the formula
\begin{equation}
(x\cdot f)(z_{1},\cdots,z_{k}):=x\cdot(f(z_{1},\cdots,z_{k}))-\sum_{i=1}^{k}f(z_{1},\cdots z_{i-1},[x,z_{i}],\cdots,z_{k})\,,\label{eq2.5}
\end{equation}
where $x,z_{1},\dots,z_{k}\in L$ and $f\in C^{k}(L,V)$.

\begin{exam}{\rm
Consider $k=1$ and $C^{1}(L,V)=\Hom_{\C}(L,V)$. For $x,y,z\in L$, the formula (\ref{eq2.5}) reads to
$(x\cdot f)(z)=x\cdot f(z)-f([x,z]).$ Thus $([x,y]\cdot f)(z) = [x,y]\cdot f(z)-f([[x,y],z])=x\cdot (y\cdot f(z))-y\cdot (x\cdot f(z))+\omega(x,y)f(z)-f([[x,y],z])$. On the other hand, we note that
\begin{eqnarray*}
x\cdot (y\cdot f)(z)&=&x\cdot((y\cdot f)(z))-(y\cdot f)([x,z])\\
&=&x\cdot(y\cdot f(z))-x\cdot f([y,z])-y\cdot f([x,z])+f([y,[x,z]])
\end{eqnarray*}
and $$y\cdot (x\cdot f)(z)=y\cdot (x\cdot f(z))-y\cdot f([x,z])-x\cdot f([y,z])+f([x,[y,z]])\,.$$
Thus
\begin{eqnarray*}
&&([x,y]\cdot f)(z)-x\cdot (y\cdot f)(z) +y\cdot (x\cdot f)(z)-\omega(x,y)f(z)\\
 &=& f([x,[y,z]])-f([[x,y],z])-f([y,[x,z]])\\
&=&-f(\omega(y,z)x+\omega(x,y)z+\omega(z,x)y)\,,
\end{eqnarray*}
which does not vanish in general, unless $L$ is a Lie algebra. 
}\end{exam}

\section{Indecomposable  modules} \label{sec3}

In this section we study indecomposable modules of some three-dimensional  $\omega$-Lie algebras and give a proof of Theorem \ref{mainthm3}.

Let $L\in\LL$ be a  three-dimensional non-Lie $\omega$-Lie algebra over $\C$ with a basis $\{x,y,z\}$.
It follows from \cite[Theorem 2]{CLZ14} that there always exists a two-dimensional Lie subalgebra $\g\subset L$, spanned by $y$ and $z$ such that $[y,z]=z$. Define $\h$ to be the subspace spanned by  $z$. Clearly, $\g$ is isomorphic to the unique  two-dimensional nonabelian Lie algebra over $\C$ and $\h$ can be viewed as an abelian Lie algebra.  Throughout this section we assume that the element $z$ belongs to $\ker(\omega)$; namely, 
$L\in\{L_{1},A_{\alpha}\}$.

Suppose $V$ is a finite-dimensional indecomposable  $L$-module. Since $V$ is also an $\h$-module, there exists a finite set $\{\lam_{1},\dots,\lam_{k}\}$ of weights of $\h$ such that
\begin{equation}
V=\bigoplus_{i=1}^{k}V_{\lam_{i}}\,,\label{eq3.1}
\end{equation}
where $V_{\lam_{i}}:=\{v\in V\mid\textrm{for each }h\in \h,\textrm{ there exists }n_{h}\textrm{ such that }
(h_{V}-\lam_{i}(h) 1)^{n_{h}}(v)=0\}\neq \{0\}$. Further, these $V_{\lam_{i}}$ are $\h$-modules; see \cite[Theorem 2.9]{Car05}.

Note that $\h\subset\g\subset L$ and $V$ is also a $\g$-module.  With above notations and conventions, we obtain several helpful lemmas.

\begin{lemma}\label{lem3.1}
For $1\leqslant i\leqslant k$, $V_{\lam_{i}}$ is a $\g$-module.
\end{lemma}

\begin{proof}
It suffices to show that  $y\cdot v\in V_{\lam_{i}}$ for all $y\in \g$ and $v\in V_{\lam_{i}}$. Consider the Lie algebra
$\g$ and the $\g$-module $V$. Since $\omega(y,z)=0$ in $L$, an analogous argument with \cite[Proposition 2.7]{Car05} implies that  for $h\in\h, \lam_{i}(h)\in\C$ and $v\in V_{\lam_{i}}$, we have
\begin{equation}
(h_{V}-\lam_{i}(h)1)^{n}(y\cdot v)=\sum_{j=0}^{n}{n\choose j}((\ad_{h})^{j}(y))(h_{V}-\lam_{i}(h)1)^{n-j}(v) \label{eq3.2}
\end{equation}
for $n\in\Z^{+}$. Note that $h=az$ for some $a\in\C$ and $[y,z]=z$. Setting $n=n_{h}+1$ in Eq. (\ref{eq3.2}) we see that
$(h_{V}-\lam_{i}(h)1)^{n_{h}+1}(y\cdot v)=0$. This means  $y\cdot v\in V_{\lam_{i}}$ and thus $V_{\lam_{i}}$ is a $\g$-module.
\end{proof}

Let $\D:=(L\oplus V,\Omega)$ be the semi-direct product of an $\omega$-Lie algebra $(L,\omega)$ and an $L$-module $V$, where $\Omega$ extends $\omega$ trivially; see \cite[Proposition 6.3]{CZZZ18} for the definition of the semi-direct product of an $\omega$-Lie algebra and its module.

\begin{lemma}\label{lem3.2}
There is an abelian Lie subalgebra $H$ of $\D$ such that $H\subseteq \ker(\Omega)$ and $\dim(H)>1$.
\end{lemma}

\begin{proof}
If $V$ is a trivial $\h$-module, i.e., $z\cdot v=0$ for all $v\in V$, then $H=\h\oplus V$ is what we want.
Now assume that $V$ is a nontrivial $\h$-module and consider the Lie subalgebra $\g\oplus V$ of $\D$.
We observe that $\g\oplus V$ is a soluble Lie algebra, thus $[\g\oplus V,\g\oplus V]$ is nilpotent. Since $[y,z]=z$, we have
$\h\oplus\{0\}\subseteq [\g\oplus V,\g\oplus V]\subseteq \h\oplus V.$ As $V$ is not a trivial $\h$-module, we can find a vector
$v_{0}\in V$ such that $z\cdot v_{0}\neq 0$. Thus $(0,z\cdot v_{0})=[(z,0),(z,v_{0})]\in [\g\oplus V,\g\oplus V]$ but not in $\h\oplus\{0\}$. This implies that $\dim([\g\oplus V,\g\oplus V])>\dim(\h)=1$. Let $V'\subseteq V$ be the subspace such that $[\g\oplus V,\g\oplus V]=\h\oplus V'$. Then $\dim(V')\geqslant 1$.
By Engel's theorem, $\ad_{(z,0)}: \h\oplus V'\longrightarrow \h\oplus V'$ is nilpotent, and it also restricts to a nilpotent linear map on $V'$.
We use $V_{1}$ to denote the kernel of $\ad_{(z,0)}$ in $V'$. Then $V_{1}\neq\{0\}$ and so $\dim(V_{1})\geqslant 1$.
Note that for any $v\in V_{1}$, the fact that $0=\ad_{(z,0)}(0,v)=[(z,0),(0,v)]=(0,z\cdot v)$ implies $z\cdot v=0$, thus the action of $\h$ on $V_{1}$ is trivial.
Let $H=\h\oplus V_{1}$. Observe that $H$ is an abelian Lie subalgebra of $\D$ such that $H\subseteq \ker(\Omega)$ and $\dim(H)>1$. The proof is completed.
\end{proof}

\begin{lemma}\label{lem3.3}
Let $\ad:\D\longrightarrow\D$ be the adjoint map. Then
\begin{eqnarray}
&&\sum_{j=0}^{n}{n\choose j} \left[ (\ad_{h}+\alpha 1)^{n-j}(u), (\ad_{h}+\beta 1)^{j}(v)\right]\nonumber \\
&=&(\ad_{h}+(\alpha+\beta)1)^{n}([u,v])-n(\alpha+\beta)^{n-1}\Omega(u,v)h\label{eq3.3}
\end{eqnarray}
for all $n\in\Z^+, u,v\in \D,h\in H$ and $\alpha,\beta\in\C$.
\end{lemma}

\begin{proof}
We apply \cite[Lemma 4.4]{Zus10}  for $\D=(L\oplus V,\Omega)$ with $H$ defined in Lemma \ref{lem3.2}.
\end{proof}

We identify $L$ with $L\oplus \{0\}$ and identify $V$ with $\{0\}\oplus V$ in $\D$. With this two identifications, we are working on $\D$. In Eq. (\ref{eq3.3}), setting $\alpha=0, h\in \h=\h\oplus\{0\}$, $u=x\in L$ and $v\in V$, we obtain the following lemma.

\begin{lemma} For any $n\in\Z^+$ and $\beta\in\C$,
\begin{equation}
(\ad_{h}+\beta1)^{n}([x,v])=\sum_{j=0}^{n}{n\choose j} \left[ (\ad_{h})^{n-j}(x), (\ad_{h}+\beta 1)^{j}(v)\right]\,.\label{eq3.4}
\end{equation}
\end{lemma}

Finally, we prove the following key lemma.

\begin{lemma}\label{lem3.5}
$V_{\lam_{i}}$ is an $L$-module for $1\leqslant i\leqslant k$.
\end{lemma}

\begin{proof}
By Lemma \ref{lem3.1} it suffices to show that $x\cdot v\in V_{\lam_{i}}$ for all $v\in V_{\lam_{i}}$. We observe that
$[x,v]=[(x,0),(0,v)]=(0,x\cdot v)=x\cdot v$ and for  $w\in V$, $(\ad_{h}+\beta1)(w)=\ad_{h}(w)+\beta 1(w)=[h,w]+\beta1(w)=
[(h,0),(0,w)]+\beta1(0,w)=(0,h\cdot w)+(0,\beta1 (w))=(h_{V}+\beta1)w$. Thus
$(\ad_{h}+\beta1)^{n}(w)=(h_{V}+\beta 1)^{n}(w)$ for all $w\in V$ and $n\in\Z^+$.
These observations, together with setting  $\beta=-\lam_{i}(h)$ in Eq. (\ref{eq3.4}), imply that
\begin{equation}
(h_{V}-\lam_{i}(h)1)^{n}(x\cdot v)=\sum_{j=0}^{n}{n\choose j} \left[ (\ad_{h})^{n-j}(x), (h_{V}-\lam_{i}(h) 1)^{j}(v)\right]\,.\label{eq3.5}
\end{equation}
Recall that $h=a z$ for some $a\in\C$ and $[z,[z,[z,x]]]=0$ in $L$. Thus $\ad_{h}^{j}(x)=0$ for $j\geqslant 3$.
Taking $n=n_{h}+2$ in Eq. (\ref{eq3.5}), we obtain $(h_{V}-\lam_{i}(h)1)^{n_{h}}(x\cdot v)=0$. Hence, $V_{\lam_{i}}$ is an $L$-module.
\end{proof}

An important consequence has been derived.

\begin{coro}\label{coro3.6}
$k=1$ in Eq. (\ref{eq3.1}).
\end{coro}

\begin{proof}
As $V$ is indecomposable, Lemma \ref{lem3.5} implies $k=1$.
\end{proof}

Suppose $n\in\Z^{+}$ and $M_{n}(\C)$ denotes the $n^{2}$-dimensional vector space of all $n\times n$ matrices over $\C$. Let $N_{n}(\C)$ be the set of all nilpotent matrices in $M_{n}(\C)$ and $D_{n}(\C)=\{\lam \mathrm{I_n}\mid \lam\in\C\}$ be
the subspace spanned by the identity matrix $\mathrm{I_n}$ in $M_{n}(\C)$. Clearly, $N_{n}(\C)\cap D_{n}(\C)=\{0\}$.
Define $$P_{n}(\C):=D_{n}(\C)\times N_{n}(\C)\,.$$ There exists a natural conjugacy action of the general linear group $\GL(n,\C)$ on $P_{n}(\C)$ given by
\begin{equation}
\sigma (\lam \mathrm{I_n}, A):=(\sigma (\lam \mathrm{I_n})\sigma^{-1}, \sigma A \sigma^{-1})=(\lam \mathrm{I_n}, \sigma A \sigma^{-1})\,, \label{ }
\end{equation}
where $\sigma\in \GL(n,\C)$, $\lam\in\C$ and $A\in N_{n}(\C)$.

We use $\R_{n}^0(\C)$ to denote the set of all indecomposable $L$-modules on $\C^{n}$ such that the actions of $x$ and $y$ on $\C^n$  are determined by the action of $z$. Let $\B_{n}(\C)$ be the set of all $\h$-modules on $\C^{n}$ and $\A_{n}(\C)$ be the subset of $\B_{n}(\C)$ consisting of all $\h$-modules for which the resulting matrix of $z$ on $\C^{n}$ can be written as the sum of two matrices from the components of $P_{n}(\C)$.

\begin{prop}\label{prop3.7}
There exists an injective map $\phi$ from $\R_{n}^0(\C)$ to $\A_{n}(\C)$.
\end{prop}

\begin{proof}
For each $V\in \R_{n}^0(\C)$,  it is also an $\h$-module. Corollary \ref{coro3.6} shows that $V$ is isomorphic to some $V_{\lam}$ for $\lam\in\Hom(\h,\C)$. Note that $\dim(\h)=1$ and $\h$ is spanned by $z$, so $\lam$ is determined by the complex number $\lam(z)$. Since $z_{V_{\lam}}-\lam(z)\mathrm{I_n}\in N_{n}(\C)$, we have $z_{V_{\lam}}=\lam(z)\mathrm{I_n}+(z_{V_{\lam}}-\lam(z)\mathrm{I_n})$, where $(\lam(z)\mathrm{I_n}, z_{V_{\lam}}-\lam(z)\mathrm{I_n})\in P_{n}(\C)$. Now we define
$$\phi:\R_{n}^0(\C)\longrightarrow\A_{n}(\C)$$ by $V\mapsto \phi(V)$, where $\phi(V)$ is determined uniquely by $z_{\phi(V)}=\lam(z)\mathrm{I_n}+(z_{V_{\lam}}-\lam(z)\mathrm{I_n})$. For any $V_{1},V_{2}\in  \R_{n}^0(\C)$, there exist
$\lam_{1},\lam_{2}\in\Hom(\h,\C)$ such that $V_{i}=V_{\lam_{i}}$ for $i=1,2$. As $\lam(z)\mathrm{I_n}$ and $z_{V_{\lam}}-\lam(z)\mathrm{I_n}$ are the semisimple and nilpotent parts respectively in the Jordan-Chevalley decomposition in $z_{\phi(V)}$, the uniqueness of the decomposition implies that if $z_{\phi(V_{1})}=z_{\phi(V_{2})}$, then $\lam_{1}=\lam_{2}$. Thus $V_{1}=V_{\lam_{1}}=V_{\lam_{2}}=V_{2}$.
This means that $\phi$ is injective.
\end{proof}

\begin{prop}\label{prop3.8}
There exists a bijection between $\A_{n}(\C)$ and $P_{n}(\C)$. Moreover,
the equivalence classes in $\A_{n}(\C)$ are in one-to-one correspondence with  the conjugacy classes in $P_{n}(\C)$.
\end{prop}

\begin{proof}
Since $\h$ is one-dimensional and spanned by $z$, any $\h$-module $V$ in $\A_{n}(\C)$ is determined by the matrix $z_{V}=\lam(z)\mathrm{I_n}+(z_{V_{\lam}}-\lam(z)\mathrm{I_n})$, where $(\lam(z)\mathrm{I_n}, z_{V_{\lam}}-\lam(z)\mathrm{I_n})\in P_{n}(\C)$. If $V\in \A_{n}(\C)$, then $\varphi(V):=(\lam(z)\mathrm{I_n}, z_{V_{\lam}}-\lam(z)\mathrm{I_n})$ gives rise to a map from $\A_{n}(\C)$ to $P_{n}(\C)$. Conversely, as $\h$ is one-dimensional, any matrix $B\in M_{n}(\C)$ could define an $\h$-module $V_{B}$ by $z_{V_{B}}=B$. If $(\lam(z)\mathrm{I_n}, B-\lam(z)\mathrm{I_n})\in P_{n}(\C)$, then $V_{B}\in \A_{n}(\C)$. Let $\varphi': P_{n}(\C)\longrightarrow \A_{n}(\C)$ be the map given by $\varphi'(B)=V_{B}$. Clearly,
$\varphi\circ\varphi'=1_{P_{n}(\C)}$ and $\varphi'\circ\varphi=1_{\A_{n}(\C)}$. Hence, $\varphi$ is a bijection between $\A_{n}(\C)$ and $P_{n}(\C)$.
Note that $V_{1}$ is equivalent to $V_{2}$ in $\A_{n}(\C)$ if and only if $z_{V_{1}}$ and $z_{V_{2}}$ are similar, if and only if
$z_{V_{1}}$ is conjugate with $z_{V_{2}}$ in $P_{n}(\C)$. This proves the second statement.
\end{proof}

\begin{proof}[Proof of Theorem \ref{mainthm3}]
Combining Propositions \ref{prop3.7} and \ref{prop3.8}, together with the fact that if $V_{1}$ is equivalent to $V_{2}$ in $\R_{n}(\C)$ then $\phi(V_{1})$ and $\phi(V_{2})$ are also equivalent in $\A_{n}(\C)$, we see that
the actions of $z\in L$ in two equivalent representations can be parameterized by the complex field and
conjugacy classes of nilpotent $n\times n$-matrices. By the nonzero generating relations in $L$, we see that the actions of 
$x$ and $y$ on $\C^n$ can be determined by finitely many polynomial equations. Thus an arbitrary action of $L$ on $\C^n$ 
can be determined by a complex number, a nilpotent matrix and two elements of an affine variety. 
This completes the proof.
\end{proof}

As an application, we conclude with the following example.

\begin{exam}{\rm
We can completely determine all 2-dimensional indecomposable $L_{1}$-modules. Suppose $V$ is  such a module.  Recall that any $2\times 2$ nilpotent matrix is similar to $(\begin{smallmatrix}
      0&0    \\
      0&0
\end{smallmatrix})$ or $(\begin{smallmatrix}
      0&1    \\
      0&0
\end{smallmatrix})$. 

For the first case, we  may assume  $$z_{V}=\begin{pmatrix}
      a&    0\\
     0 &a
\end{pmatrix}, y_{V}=\begin{pmatrix}
     b_{1} &   b_{3} \\
     b_{2} &  b_{4}
\end{pmatrix} \textrm{ and } x_{V}=\begin{pmatrix}
     c_{1} &   c_{3} \\
     c_{2} &  c_{4}
\end{pmatrix}$$ with respect to a basis $\{e_{1},e_{2}\}$ of $V$, where $a, b_i, c_i\in \C, 1\leqslant i\leqslant 4$.
By Eq. (\ref{eq2.1}), we obtain two subcases: \begin{enumerate}
\item  $z_{V}=\begin{pmatrix}
      0&    0\\
     0 &0
\end{pmatrix}$, $y_{V}=\begin{pmatrix}
     1 &   1 \\
     0&  1
\end{pmatrix}$ and $x_{V}=\begin{pmatrix}
     c+1 &   b \\
     0&  c
\end{pmatrix}$;
\item $z_{V}=\begin{pmatrix}
      0&    0\\
     0 &0
\end{pmatrix}$, $y_{V}=\begin{pmatrix}
     1 &   0 \\
     0&  1
\end{pmatrix}$ and $x_{V}=\begin{pmatrix}
     c &   1 \\
     0&  c
\end{pmatrix}$,
\end{enumerate}
 where $b, c\in \C$. 

For the second case, we  assume  $$z_{V}=\begin{pmatrix}
      a&    1\\
     0 &a
\end{pmatrix},   y_{V}=\begin{pmatrix}
     b_{1} &   b_{3} \\
     b_{2} &  b_{4}
\end{pmatrix} \textrm{ and }  x_{V}=\begin{pmatrix}
     c_{1} &   c_{3} \\
     c_{2} &  c_{4}
\end{pmatrix}$$ with respect to a basis $\{e_{1},e_{2}\}$ of $V$, where $a, b_i, c_i\in \C, 1\leqslant i\leqslant 4$. A direct calculation leads to $3/2=b_{1}=1$, which is a contradiction. It also shows that the map $\phi$ in Proposition \ref{prop3.7} is not surjective. 
}\end{exam}

\section{Tailed derivations of Lie algebras}

The last section is mainly to study relations between one-dimensional $\omega$-extensions of a Lie algebra $\g$ and tailed derivations of $\g$, focusing on fundamental properties and examples on tailed derivations of Lie algebras and giving a proof of Theorem  \ref{mainthm4}.

 \begin{prop} \label{prop4.1}
Let $A$ be a nonassociative algebra.  Then
$\TDer(A)$ is a Lie subalgebra of $\gl(A)$.
\end{prop}

\begin{proof}
Suppose $D, T\in\TDer(A)$ are arbitrary tailed derivations.  For $y,z\in A$, we have
\begin{eqnarray*}
&&(D+T)([y,z])\\
&=&D([y,z])+T([y,z])\\
&=&[D(y),z]+[y,D(z)]+d_{z}y-d_{y}z+[T(y),z]+[y,T(z)]+t_{z}y-t_{y}z\\
&=&[(D+T)(y),z]+[y,(D+T)(z)]+(d_{z}+t_{z})y-(d_{y}+t_{y})z\,,
\end{eqnarray*}
where $d_{y},t_{y},d_{z},t_{z}\in \C$. For $a\in \C$, we see that
$(aD)([y,z])=a(D[y,z])=a([D(y),z]+[y,D(z)]+d_{z}y-d_{y}z)=[(aD)(y),z]+[y,(aD)(z)]+ad_{z}y-ad_{y}z$. This means that $\TDer(A)$ is a subspace of $\gl(A)$. To show $\TDer(A)$ is a Lie subalgebra of $\gl(A)$, it suffices to show that
$[D,T]=DT-TD$ is also a tailed derivation. Indeed, since
\begin{eqnarray*}
DT([y,z])&=&D([T(y),z]+[y,T(z)]+t_{z}y-t_{y}z)\\
&=&[DT(y),z]+[T(y),D(z)]+d_{z}T(y)-d_{T(y)}z+[D(y),T(z)]\\
&&+[y,DT(z)]+d_{T(z)}y-d_{y}T(z)+t_{z}D(y)-t_{y}D(z)\,,\\
TD([y,z])&=&T([D(y),z]+[y,D(z)]+d_{z}y-d_{y}z)\\
&=&[TD(y),z]+[D(y),T(z)]+t_{z}D(y)-t_{D(y)}z+[T(y),D(z)]\\
&&+[y,TD(z)]+t_{D(z)}y-t_{y}D(z)+d_{z}T(y)-d_{y}T(z)\,,
\end{eqnarray*}
we have
\begin{equation}
[D,T]([y,z])=[[D,T](y),z]+[y,[D,T](z)]+(d_{T(z)}-t_{D(z)})y-(d_{T(y)}-t_{D(y)})z\,.\label{eq4.1}
\end{equation}
Note that $d_{T(-)}-t_{D(-)}=(d\circ T-t\circ D)(-)$ is a linear form of $A$. Thus $[D,T]$ is a tailed derivation of $A$.
This shows that $\TDer(A)$ is a Lie algebra.
\end{proof}

\begin{exam}\label{exam4.2}
{\rm
Let $\g$ be the two-dimensional nonabelian Lie algebra defined by $[y,z]=z$ and $D=\left(\begin{smallmatrix}
      a&c   \\
      b&e
\end{smallmatrix}\right)$ be a linear map on $\g$ with respect to the basis $\{y,z\}$, where $a,b,c,e\in\C$. A direct calculation shows that if $D\in\Der(\g)$, then $a=c=0$. Thus $\dim\Der(\g)=2$. Moreover, consider the linear form $d$ which
sends $y$ to $a$ and $z$ to $c$. Then  together with the linear form $d$, every $D=\left(\begin{smallmatrix}
      a&c   \\
      b&e
\end{smallmatrix}\right)$ is a tailed derivation of $\g$.
This means that $\TDer(\g)=\gl_{2}(\C)$, strictly containing $\Der(\g)$. 
}\end{exam}

\begin{prop}
Let $L$ be an $\omega$-Lie algebra with a nonzero proper ideal $\g$. Suppose $L=\g\oplus \h$ denotes a decomposition of vector spaces. Then $\ad_{x}$ restricted to $\g$ is a tailed derivation of $\g$ for all $x\in\h$.
\end{prop}

\begin{proof}
By \cite[Corollary 3.2]{Zus10} we see that $\g$ is a Lie algebra; thus $\omega(\g,\g)=0$. Now suppose $y,z\in \g$ and $x\in\h$ are arbitrary elements.
The $\omega$-Jacobi identity implies that $\ad_{x}([y,z])=[\ad_{x}(y),z]+[y,\ad_{x}(z)]+\omega(x,z)y-\omega(x,y)z.$
Thus $\ad_{x}$ restricted to $\g$ is a tailed derivation of $\g$.
\end{proof}

\begin{lemma}\label{lem4.4}
Let $\g$ be a nonzero Lie subalgebra of an $\omega$-Lie algebra $L$ of dimension $\dim(L)-1$. Let $x\in L\setminus\g$ be an arbitrary nonzero vector. Then $\ad_{x}$ restricted to $\g$ is a tailed derivation of $\g$.
\end{lemma}

\begin{proof}
Let $y,z\in\g$. As  $\omega(y,z)=0$, it follows from the  $\omega$-Jacobi identity that $\ad_{x}([y,z])=
[\ad_{x}(y),z]+[y,\ad_{x}(z)]+\omega(x,z)y-\omega(x,y)z$. Clearly, $\omega(x,-)$ is a linear form of $\g$. Hence $\ad_{x}$ is a tailed derivation of $\g$.
\end{proof}

Now we are ready to prove Theorem \ref{mainthm4}.

\begin{proof}[Proof of Theorem \ref{mainthm4}]
Let $L_{x}\in\ext(\g)$ be a one-dimensional $\omega$-extension of $\g$ through $\C x$. Lemma \ref{lem4.4} shows that the adjoint map
$\ad_{x}:L_{x}\longrightarrow L_{x}$ restricted to $\g$ is an element of $\TDer(\g)$. We can define a map $\varphi: \ext(\g)\longrightarrow\TDer(\g)$ by carrying $L_{x}$ to $\ad_{x}|_{\g}$. Conversely, if $D$ is a tailed derivation of $\g$, then there exists a linear form $d$ of $\g$  such that $D([y,z])=[D(y),z]+[y,D(z)]+d_{z}y-d_{y}z$ for all $y,z\in\g$. We define an $\omega$-Lie algebra $L_{x}=\g\oplus\C x$  by
\begin{eqnarray*}
L_{x} :  [x,y]=D(y), [x,x]=0 \textrm{ and }  \omega(x,y)=d_{y}, \omega(x,x)=0
\end{eqnarray*}
for all $y\in\g$; the remaining bracket product $[y,z]$ in $L_{x}$ matches with that in $\g$ and $\omega(y,z)=0$ for all $y,z\in\g$. Note that $d_{y}$ only depends upon $y$ so $\omega(x,y)=d_{y}$ does make sense.
Thus $L_{x}$ is a well-defined $\omega$-Lie algebra. We also define a map $\phi:\TDer(\g)\longrightarrow\ext(\g)$ by $\phi(D)=L_{x}$.
Furthermore, note that $\g$ is an ideal of $L_{x}$ and by the previous construction we see that $\phi\circ\varphi=1_{\ext(\g)}$ and $\varphi\circ\phi=1_{\TDer(\g)}$.
This completes the proof.
\end{proof}

Theorem \ref{mainthm4} indicates that the problem of finding all one-dimensional $\omega$-extensions of a Lie algebra $\g$ could be transformed to calculate tailed derivations of $\g$.
As a direct application, the following example illustrates how to determine all one-dimensional $\omega$-extensions of three-dimensional simple Lie algebra $\sll_{2}(\C)$.

\begin{exam}{\rm
Suppose that $\sll_{2}(\C)$ has a basis $\{e_{1},e_{2},e_{3}\}$ with $[e_{1},e_{2}]=-e_{1}, [e_{1},e_{3}]=2e_{2}$ and $[e_{2},e_{3}]=-e_{3}$. A tedious but direct calculation shows that $\Der(\sll_{2}(\C))=\TDer(\sll_{2}(\C))$ has dimension 3 and  the element $D\in\Der(\sll_{2}(\C))$ is of the  form:
$$D=\begin{pmatrix}
      a&  b&0  \\
      -2c&0&-2b\\
      0&c  &-a
\end{pmatrix}\,,$$
where $a,b,c\in\C$. Hence, any one-dimensional $\omega$-extension of  $\sll_{2}(\C)$ can be determined by at most three parameters.
}\end{exam}

We  present an example of a non-Lie $\omega$-Lie algebra that can be obtained by a Lie algebra $\g$ and a tailed derivation $D$ of $\g$.

\begin{exam}\label{exam4.6}
{\rm
Let $\g$ be the two-dimensional nonabelian Lie algebra defined by $[y,z]=z$ and $D=\left(\begin{smallmatrix}
      1&0    \\
      0&0
\end{smallmatrix}\right)$ be a linear map on $\g$ with respect to the basis $\{y,z\}$. Let $\{y^*, z^*\}$ be the dual basis. Then $y^*:\g\longrightarrow\C$ is a linear form
such that $D$ becomes a tailed derivation of $\g$. By the construction in the proof of Theorem \ref{mainthm4} we eventually derive
a three-dimensional non-Lie $\omega$-Lie algebra which is actually the $\omega$-Lie algebra $L_{1}$ in Example \ref{exam2.1}.
The $\omega$-Lie algebra $L_{2}$ in Example \ref{exam2.1}  can also be  obtained in a similar way.
}\end{exam}

\section*{Acknowledgments}

This research was partially supported by NNSF of China (No. 11301061).
The author thanks the anonymous referee for his/her careful reading and suggestions for the first draft of this article.


\end{document}